\documentclass[12pt,a4paper]{amsart}
\usepackage{amsfonts}
\usepackage{amsthm}
\usepackage{amsmath}
\usepackage{amscd}
\usepackage[latin2]{inputenc}
\usepackage{t1enc}
\usepackage[mathscr]{eucal}
\usepackage{indentfirst}
\usepackage{graphicx}
\usepackage{graphics}
\usepackage{pict2e}
\usepackage{epic}
\numberwithin{equation}{section}
\usepackage[margin=2.9cm]{geometry}
\usepackage{epstopdf}

\theoremstyle{plain}
\newtheorem{Th}{Theorem}[section]
\newtheorem{Lemma}[Th]{Lemma}
\newtheorem{Cor}[Th]{Corollary}

 \theoremstyle{definition}
\newtheorem{Def}[Th]{Definition}

\newtheorem{?}[Th]{Problem}
\newtheorem{Ex}[Th]{Example}

\begin{document}

\author[Jorge Mello]{Jorge Mello}

\address{Oakland University. mailing adress:\newline Mathematics and Science Center
Michigan.} 

\email{jorgedemellojr@oakland.edu}
\title[ Sparsity for quasi-integral points in orbits and correspondences ]
{ On Sparsity of integral points in orbits and correspondences with big iterated pullbacks }
\maketitle

\begin{abstract} We prove new unconditional results of sparsity of integral points on orbits under many maps and correspondences in arbitrary dimensions, generalizing theorems of Yasufuku(2015) and others. The main ingredients are new diophantine approximation tools and recent constructions  for correspondences.
\end{abstract}
\section{Introduction} 
Let $\phi:X \rightarrow X$ be a self rational map on a projective variety defined over a number field $k$ and denote by $\phi^{(n)}$ the $n$-iterate of $\phi$. We investigate sparsity in the currently very active field of study about the frequency of integral points lying in the forward orbits $$\mathcal{O}_\phi(P):=\{P,\phi(P),\phi^{(2)}(P),...\}$$ with $P \in X(k)$ in arithmetic dynamics, particularly  in higher dimensions. In dimension $1$, finiteness of algebraic integers in orbits of rational functions was proven in a key breakthrough result of Silverman \cite{S1}, who showed the only exceptions are maps whose second iterate is a polynomial. This was made quantitative and extended to "quasi-integral points"  \cite{HS}, to semigroup dynamics \cite{M2, S1}, and also consistently described over function fields in \cite{CHT,HSW,M1} recently, just to cite a few. As in the classical cases of curves and abelian varieties, the main tools to show sparsity of integral points have been mostly powerful diophantine approximation results like Roth's theorem or Schmidt's subspace theorem, although in higher dimensions the problem is much more complicated. Zariski non-density of integral points in varieties of general type should be true according the Bombieri-Lang conjecture. This is also a consequence of a  deep diophantine approximation conjecture due to Vojta that generalizes Roth and Schmidt, and also implies the ABC conjecture.  Some of the first generalizations of Silverman's theorem to higher dimensions are conditional to Vojta's conjecture, as in the following nice sparsity theorem of Yasufuku on projective spaces
\begin{Th}(Yasufuku)\cite{Y}
Let $\phi: \mathbb{P}^N \rightarrow \mathbb{P}^N$ be a morphism defined over $\mathbb{Q}$ of degree $d\geq 2$ and $D$ be a divisor of $\mathbb{P}^N$. Suppose there exists $m>0$ such that the normal-crossings subdivisor $D_{nc}^{(m)}$ of the pullback $(\phi^{(m)})^*(D)$ satisfies Vojta's conjecture and $\deg D_{nc}^{(m)}>N+1$. Then, for any $P \in \mathbb{P}^N(\mathbb{Q})$, the set $\mathcal{O}_\phi(P) \cap \mathbb{P}^N(\mathbb{Z} \setminus |D|) $ is not Zariski dense.
\end{Th} The theorem holds more generally over a number field $k$ and replacing integral points with the set of $(D,S)$-integral points 
$$
\left\{P \in \mathbb{P}^N(k)| \sum_{v \notin S}\lambda_v(D,P) \leq h(D,P) \right\}
$$where $S$ is a finite set of places of $k$, $c$ is a positive constant and $\lambda_v(D,.)$ is a Weil local height function and $h(D,.)$ is a global Weil height. 

It also agrees with Silverman's result in one variable. Other recent higher dimensional dynamical analogues of Siegel's theorem for curves like the previous one can be found in recent works \cite{CSTZ, GN, Mat, Y, Y2}, either relying on the conjectures of Vojta, or on new generalizations of the Schmidt's subspace theorem or an effective Runge-type result. Here, among other things we start using a new generalization of Schmidt's subspace theorem due to He and Ru \cite{HR} to obtain an unconditional version of Yasufuku's Theorem for projective varieties as follows
\begin{Th} (See Theorem 5.1)
Let $X$ be a projective variety def. over $k$ of dimension $n$, $S \subset M_k$ be a finite set of places and $\phi$ be an endomorphism of $X$ def. over $k$.  Suppose there exist $ m\geq 0$, a nontrivial effective divisor $D$ and $A$ an ample divisor def. over $k$ of $X$  such that
\begin{itemize}
\item[\emph{(i)}] $(\phi^{(m)})^*(D) = D_{1}  + D_{2} + \cdots +D_{ q} + D'$ is a sum of effective divisors def. over $k$ where $D_{1}, \ldots, D_{ q}$ are in $l$-subgeneral position for some $l\geq n$.
\item[\emph{(ii)}] There exists $c_j \in \mathbb{Q}$ such that
$ \sum_{j=1}^{q} c_{j}{D_{j}} - (l-n+1)(n+1) A$ is big  and   $A-c_{j}D_{j}$ is $\mathbb{Q}$-nef for all $1 \leq  j \leq q$.
\end{itemize}
Then for all sufficiently small $ \epsilon >0$ and all $P\in X(k)$ the set
$$
\left\{\phi^{(r)}(P):r \in \mathbb{N},  \sum_{v\notin S} \lambda_v(D, \phi^{(r)}(P)) \le \epsilon h(A,\phi^{(r)}(P))\right\}
$$ is not Zariski dense.
\end{Th} This theorem recovers and generalizes  Yasufuku's results when his $D_{(nc)}^m$ is a union of lines and his theorem becomes  unconditional due to Schmidt's subspace theorem. Yasufuku's theorem can be made quantitative with respect to bounding the largest number of iterations occurring in an integral iterate (even if we allow the base field to vary with bounded degree) and also uniform not depending on the initial point of the orbit. This was recently achieved by Yasufuku in \cite{Y2} using new diophantine approximation ingredients. His uniformity result replacing points in a given number field with algebraic points of bounded degree provides a version of the univariate results of Gunther and Hindes  \cite{GH}  for $\mathbb{P}^N.$ Similarly, we prove a non-quantitative algebraic points analogue result  for projective varieties assuming bigness of some iterate pullbacks of a base divisor as suggested in \cite[Remark 8]{Y2}. For dynamics with one map, we state that  below. The symbol $\equiv$ denotes numerical equivalence of divisors.
\begin{Th} (See Thm.'s 5.4 and 5.5)
Let $X$ be a projective variety def. over $k$ of dimension $n$, $S \subset M_k$ be a finite set of places, and $\phi$ be an endomorphism of $X$ def. over $k$. Let $\delta \in \mathbb{N}^{>0}$.  Suppose there exist a positive integer $ m$ and a nontrivial effective divisor $D$ and $A$ an ample divisor def. over $k$ of $X$  such that
\begin{itemize}
\item[\emph{(i)}] $(\phi^{(m)})^*(D) = D_{1}  + D_{2} + \cdots +D_{ q} + D'$ is a sum of effective divisors def. over $k$ where $D_{1}, \ldots, D_{q}$ are in $l$-subgeneral position with $l\geq n$. 
\item[\emph{(ii)}] There exist integers $d_i>0$ such that
$ \sum_{j=1}^{q} \frac{D_{j}}{d_{j}} - C(l,n,\delta) A$ is big for all $ i$ and $D_{j}\equiv d_{j}A $  for all $j$ where $C(l,n,\delta)$ is a positive constant given in \cite{Le}.
\end{itemize}
Then for all sufficiently small $ \epsilon>0$ the set
\begin{equation*}
\left\{\phi^{(r)}(P): P \in X(\overline{k}),   [k(P):k]\leq \delta, r \geq m, \sum_{v\notin S}\sum_{\substack{w\mid v\\ w \in M_{k(P)}}} \lambda_w(D, \phi^{(r)}(P)) \le \epsilon h(A,\phi^{(r)}(P))\right\}
\end{equation*}is not Zariski dense. Moreover, if $X=\mathbb{P}^2, \delta=2$, the $D_j$'s are lines in general position and $A \sim O_{\mathbb{P}^2}(1)$, then condition (ii) can be reduced to $q>15/2$. If $X=\mathbb{P}^1$ and $\delta=1,$ the theorem holds with $q>2$.
\end{Th} 

Both theorems 1.2 and 1.3 are first proven and stated in the more general context of semigroup dynamics generated by finitely many maps(non-commuting in general), which generalize the semigroup results of \cite{M1,MY} to higher dimensions and points of bounded degree as well as the results of  \cite{GH} to higher dimensions ( as in \cite{Y2}) and to semigroups, but assuming hypothesis on $\mathbb{Q}$-divisors.

Lastly, we obtain the first versions of these results in the context of correspondences, which generalizes the scenario of deterministic dynamical systems. By $\tilde{C}$ a correspondence on $X$ we mean a subvariety
$\tilde{C} \subset X \times X$ such that the coordinate projections $x,y$ are both finite and surjective.
We  think of the points $X(k)$ as a set of vertices, with the points $\tilde{C}(k)$
defining directed edges between those vertices, which induces a graph whose union of paths generalizes the union of orbits in a classical dynamical system. In this context, we prove analogous sparsity results for similar generalized integrality sets using constructions on correspondences that were made by Ingram \cite{I} when there are divisors with large Correspondence Iterate Pullbacks (CIP). This is a concept that naturally generalizes pullbacks in dynamical systems and  that we define more precisely in Section 4.   Shortly, we say that a divisor $D^{(m)}$ of $X$ is an $m$-th \textit{ (CIP) of $D$ with respect to $\tilde{C}$}  if there are divisors $L_1,...,L_{m-1}$ of $X$ such that, making $L_0=D$ and $L_m=D^{(m)}$, we have 
 $
 y^*L_{i-1}= x^*L_i \text{ for all } i=1,...,m.
 $
 In particular, for projective spaces, we obtain
\begin{Th} (See Theorem 5.9)
Let  $\tilde{C} \subset \mathbb{P}^N \times  \mathbb{P}^N$ be a subvariety (correspondence) of $ \mathbb{P}^N$ def. over $k$, $S$ finite set of places of $k$. Let $\mathcal P$ be the scheme of paths associated to $\tilde{C}$, with associated maps $\pi : \mathcal{P} \rightarrow  \mathbb{P}^N$ and  $\sigma: \mathcal{P} \rightarrow \mathcal{P}$ as it is described in Section \ref{corresp}, and let $\delta \in \mathbb{N}^{>0}$.  
Suppose that there exist $m \in \mathbb{N}^{>0}$ and a nontrivial effective divisor $D$ defined over $k$ that has an $m$-th iterate (CIP)  $D^{(m)}$ with respect to $\tilde{C}$ such that
\begin{itemize}
\item[\emph{(i)}] $D^{(m)} = D_1+...+D_q + D^{(\text{rest})}$ is a sum of effective divisors def. over $k$  where $D_1,...,D_q$ are ample and in $l$-subgeneral position for some $l\geq n$.
\item[\emph{(ii)}]  
$  q> C(l,n,\delta).$   

  \end{itemize}
Then for all sufficiently small $\epsilon>0$, the set
\[
\left\{
\begin{aligned}
\pi(\sigma^{(r)}(P)):\quad&
P=(P_1,P_2,...)\in\mathcal{P}(\overline{k}),\quad r\geq m,\quad
[k(\pi(\sigma^{(r-m)}(P))):k]\leq\delta,\\
&\sum_{v\notin S}\sum_{\substack{w\mid v\\ w\in M_{k(P_r)}}}
\lambda_w\bigl(D,\pi(\sigma^{(r)}(P))\bigr)
\leq\epsilon h\bigl(\pi(\sigma^{(r)}(P))\bigr)
\end{aligned}
\right\}.
\]
 is not Zariski dense. Moreover, when $N=2, \delta=2,$ and the $D_j$'s are lines in general position we can replace condition (ii) above by $q>15/2$.
\end{Th} 

As another byproduct, we obtain the following novel univariate finiteness statement for integers in correspondences which in the particular setting of dynamical systems ( when $\tilde{C}$ is the graph of a rational function)  agrees with the theorems of Silverman and Gunther-Hindes, holding for correspondences having an  "iterate" with at least 3 poles,
\begin{Th} (See Corollary 5.10) Let  $\tilde{C} \subset \mathbb{P}^1 \times  \mathbb{P}^1$ be a correspondence of $ \mathbb{P}^1$ defined over a number field $k$, and $S$ be a finite set of places of $k$. Let $\mathcal P$, $\pi : \mathcal{P} \rightarrow  \mathbb{P}^1$ and  $\sigma: \mathcal{P} \rightarrow \mathcal{P}$ as in Thm. 1.4 and let $\delta \in \mathbb{N}^{>0}$.
Suppose that the divisor at infinity 
$(\infty)=[1:0]$ has an $m$-th iterate (CIP)  $D^{(m)}$ with respect to $\tilde{C}$ that is an effective $\mathbb{Q}$-divisor and whose support has more than $  C(1,1,\delta)$  points.  Denote by $z$ the affine coordinate of $\mathbb{P}^1$ so that $\mathbb{P}^1\setminus\{\infty\}=\{[z:1] :z \in \mathbb{C} \}$.
 
Then the subset 
\[
\left\{
\begin{aligned}
P_r:\ P=(P_1,P_2,...)\in\mathcal P(\overline k),\ r\ge m,\ [k(P_{r-m}):k]\le\delta,\
\sum_{v\notin S}\sum_{\substack{w\mid v\\ w\in M_{k(P_r)}}}\lambda_w(\infty,P_r)\le\varepsilon\,h(P_r)
\end{aligned}
\right\}.
\]
of $\mathbb{P}^1(\overline{k})$ is finite.

 In particular, there are only finitely many $P_r$ with $r\ge m$, $[k(P_{r-m}):k]\le\delta$ and
$P_r\neq\infty$ for which $z(P_r)$ is $S$-integer. In other words, the subset of the set above formed by its points with  $ \delta=1$ is finite  when  $q>C(1,1,1)=2.$
\end{Th} 

In the special case in which the correspondence is the graph of a rational function and the corresponding paths are the dynamical orbits of $f$, the result above recovers results of \cite{GH} and extends it to $\epsilon$-integrality as in \cite{Y2}, but also for integrality with respect to $\mathbb{Q}$-divisors.
\begin{Ex} Let $k=\mathbb{Q}$, $S=\{\infty,2\}$ and let $C\subset\mathbb{P}^1\times\mathbb{P}^1$ be parametrized by
\[
x=t^2,\quad y=\frac{(t^2-1)(t^2-2)(t^2-3)}{t}\quad\iff \quad
x\,y^2=\bigl((x-1)(x-2)(x-3)\bigr)^2 .
\]
The normalization of $C$ is  $\tilde{C}=\mathbb{P}^1_t$ def. by $t=(x-1)(x-2)(x-3)/y$. 

Take
$D=2(\infty)$. Then
\[
y^*D=2(t{=}0)+10(t{=}\infty)=x^*L_1,\qquad L_1=(0)+5(\infty),
\]
and since $y(-t)=-y(t)$, $y^*(0)=x^*\bigl((1)+(2)+(3)\bigr)$, whence
\[
D^{(2)}=(1)+(2)+(3)+\tfrac52(0)+\tfrac{25}2(\infty),\qquad
\deg D^{(2)}=18 .
\]
So $q=5>C(1,1,2)=4$, and for small $\varepsilon$ there are only finitely many vertices $P_r$ ($r\ge2$) that are
$\varepsilon$-quasi-integral with $P_{r-2}$ being quadratic over $\mathbb{Q}$. Yet $C$ itself
has infinitely many $S$-integral points:
$t=\pm2^n$ gives $\bigl(4^n,\pm(4^n-1)(4^n-2)(4^n-3)/2^n\bigr)$, i.e.
infinitely many integral vertices at distance~$1$ from a vertex that is quadratic over $\mathbb{Q}$. At level $1$, $L_1$ is
supported on only $2$ points.

\end{Ex}
In Section 2 we discuss some preliminaries on heights and Section 4 contains the needed background on correspondences. The diophantine approximation tools are stated in Section 3, and the main results are proven in Sections 5 for nef divisors and pullbacks using He and Ru's theorem,  for algebraic points and correspondences respectively.

\section{Preliminaries on heights }\label{sec:heights}

Let $k$ be a number field, and let $M_k$ be the set of places.  For the unique archimedean place of $M_\mathbb{Q}$, we use the usual absolute value on $\mathbb{R}$ as the normalized one.  For the non-archimedean place of $M_\mathbb{Q}$ corresponding to the prime $p$, we normalize the absolute value by defining the absolute value of $p$ to be $\dfrac{1}{p}$.  We then normalize absolute value corresponding to each place $v\in M_k$ by defining $|x|_v$ to be the $\frac{[k_v:\mathbb{Q}_v]}{[k:\mathbb{Q}]}$-th power of the absolute value in $v$ whose restriction to $\mathbb{Q}$ is a normalized absolute value on $\mathbb{Q}$.  We then define a Weil (global) height on the projective space $\mathbb{P}^N$ by
\[
h([a_0:\cdots : a_N] ) = \sum_{v\in M_k} \log \max_i |a_i|_v
\]
for $[a_0:\cdots : a_N]\in \mathbb{P}^N(k)$.  This becomes a well-defined function on $\mathbb{P}^N(\overline{\mathbb{Q}})$.  For a nontrivial effective divisor $D$ on $\mathbb{P}^N$ defined over $k$, we choose a homogeneous polynomial $F$ of degree $d$ with coefficients in the ring $\mathcal O_k$ of integers of $k$ so that $D$ is defined by $F$, and we define a local height function for each $v\in M_k$ by
\[
\lambda_v(D, [a_0:\cdots : a_N]) = \log \frac{(\max |a_i|_v)^d}{|F(a_0, \ldots, a_N)|_v}.
\]
This is a well-defined function on $(\mathbb{P}^N\setminus |D|)(k)$.  For a general divisor $D$ on a projective variety $X$ over $k$, we first write $D$ as a difference of two very ample Cartier divisors $D_1$ and $D_2$, where $D_i = \varphi^*(H)$ for some closed immersion $\varphi_i: X \longrightarrow \mathbb{P}^{N_i}$ and $H_i$ is a hyperplane in $\mathbb{P}^{N_i}$.  We then define Weil height to be
\[
h(D, P) = h(\varphi_1(P)) - h(\varphi_2(P)).
\]
Moreover, letting $s_{i,1},\ldots, s_{i, \ell_i}$ be a basis of global sections of the line bundle $\mathscr L(D_i)$ and choosing one rational section $s$ of $\mathscr L(D)$, we can define local height to be
\[
\lambda_v(D,P) = \max_m \min_\ell \log \left| \big(s_{1,m} \otimes (s_{2,\ell} \otimes s)^{-1}\big)(P) \right|_v,
 \]
where we evaluate at $P$ via the isomorphism $\mathscr L(D_1)\otimes (\mathscr L(D_2) \otimes \mathscr L(D))^{-1} \cong \mathscr O_X$.  This notion is well-defined in the sense that if we choose different parameters for the ample divisors or the global/rational sections, the two local height functions agree on all but finitely many places and even at those finitely many places, their difference is a bounded function.  The local and Weil height are functorial with respect to pullbacks, namely,
\[
h(\varphi^* D, P) = h(D, \varphi(P))+O(1)  \qquad \text{ and } \qquad   \lambda_v(\varphi^*D, P) = \lambda_v(D, \varphi(P)) +O(1)
\]
whenever $\varphi$ is a morphism of algebraic varieties, where equality is interpreted to mean agreement up to a bounded function (or up to bounded functions in just finitely many places).  

When $D$ and $E$ are divisors of $X$, the following additive properties also hold for every $P \in X(k)$ where the functions are defined
\[
h(D+E, P) = h(D,P)+h(E,P)+O(1)  \quad \text{ and } \quad   \lambda_v(D+E, P) = \lambda_v(D, P)+\lambda_v(E, P) +O(1)
\]
Also, if $D$ is effective, then 
\[
h(D, P) \geq O(1)  \quad \text{ and } \quad   \lambda_v(D, P) \geq O(1)
\]for all $P\in (X\setminus |D|)(k)$.

Moreover, it follows from the definitions that
\begin{equation}\label{eq:htdecompose}
h(D,P) = \sum_{v\in M_k} \lambda_v(D,P) +O(1)
\end{equation}
for all $P\in (X\setminus |D|)(k).$

For $\epsilon$ a positive real number, $S$ a finite set of places of $k$ and $D$ an effective divisor of $X$ defined over $k$,  a set of $(D,S,\epsilon)$-quasi-integral points is a set of the form
$$
\left\{P \in X(\overline{k}):   \sum_{v\notin S}\sum_{\substack{w\mid v\\ w \in M_{k(P)}}} \lambda_w(D, P) \le \epsilon h(A,P)\right\}.
$$ where $k(P)$ is the field of definition of $P$. A set of $(D,S,1)$-quasi-integral points is a set of $(D,S)$-integral points.

The result below is well known and shows that big heights and moreover ample heights are large up to a constant multiple.
\begin{Lemma}\cite[Prop.'s 10.7 and 10.9]{V} Let $M$ and $L$ be divisors of $X$, all defined over a number field $k$. If $L$ is ample, then there is a constant $C>0$ depen ding on $L$ and $M$ such that 
$$
h(M,P)\leq Ch(L,P) +O(1)
$$ for all $P \in X(\overline{k})$ where $O(1)$ depends only on $L,M$ and the choice of height functions.

On the other hand, if $L$ is big, then there exist a constant $C>0$ and a proper Zariski dense subset depending only on $L$ and $M$ such that
$$
h(M,P)\leq Ch(L,P) +O(1)
$$ for all $P \in X(\overline{k})\setminus Z$ where $O(1)$ depends only on $L,M$ and the choice of height functions.

\end{Lemma}

\section{Tools from diophantine approximation}
Here we state the diophantine approximation results we will use. First, the following recent by He and Ru.  For the definitions of general position and subgeneral position, see \cite{HR,L}
\begin{Lemma} Let $X$ be a projective variety of dimension $n$ defined over a number field  $k$ . Let $S$ be a finite set of places of $k$. Let $D_{1} , . . . , D_{q}$ be effective Cartier divisors on $X$, defined over $k$, in $l$-subgeneral position for $l \geq n$. Let $A$ be an ample Cartier divisor on $X$ and let $c_{i}$ be rational numbers such that $A -c_{i} D_{i}$ is a nef $\mathbb{Q}$-divisor for all $i$. Let $\epsilon > 0$. Then there exists a proper Zariski-closed subset $Z \subset X $such that for all points $P \in X(k) \setminus Z$,
$$
\sum_{v \in S} \sum_{i=1}^q c_{i}\lambda_v(D_{i},P)<[(l-n+1)(n+1)+\epsilon]h(A,P)+O(1).
$$
\end{Lemma}
\begin{proof}
This is a direct consequence of \cite[Corollary of the Main Theorem(pg. 140)]{HR} applied with $Y_1=D_1,...,Y_q=D_q$ and of the fact that if $D$ is an effective Cartier divisor, then the Seshadri constants therein simplify to 
$$
\epsilon_D(A)=\sup \{\gamma \in \mathbb{Q}_{\geq 0} | A-\gamma D \text{ is } \mathbb{Q}\text{-nef} \}
$$(see the discussion in \cite[Page 126]{HL})
\end{proof} The next one is a new Schmidt subspace type of theorem for algebraic points. The symbol $\equiv$ denotes numerical equivalence of divisors.
\begin{Lemma} Let $X$ be a projective variety of dimension $n$ defined over a number field  $k$ . Let $S$ be a finite set of places of $k$ and $D_{1} , . . . , D_{q}$ be effective ample Cartier divisors on $X$, defined over $k$, in $m$-subgeneral position with $m \geq n$. Let $A$ be an ample Cartier divisor on $X$ and let $d_{i}$ be positive integers such that $d_i A \equiv  D_{i}$ for all $i$ and for all $v \in S$. Let $\epsilon > 0$ and $\delta \in \mathbb{N}^{\geq 1}$. Then 
$$
\sum_{v \in S}\sum_{\substack{w\mid v\\ w \in M_{k(P)}}} \sum_{i=1}^q \dfrac{\lambda_w(D_{i},P)}{d_i}<[C(m,n,\delta)+\epsilon]h(A,P)+O(1)
$$ holds  for all points $P \in X(\overline{k}) \setminus \bigcup_i \text{Supp } D_i$ satisfying $[k(P):k]\leq \delta$, where $C(m,n,\delta)$ is a positive constant given in \cite[Page 229]{Le}.
\begin{proof}
This is \cite[Theorem 2]{Le}, which improves the bound on \cite[Theorem 5.2]{L} together with \cite[Lemma 3.4]{RW}.
\end{proof}
\end{Lemma} When $X$ is $\mathbb{P}^2$ and the $D_j$'s are lines in general position, $C(m,n,\delta)$ to 15/2 can be improved as follows
\begin{Lemma}\cite[Theorem 1.5]{L} Let $S$ be a finite set of places of a number field $k$ let $L_{1} , . . . , L_{q} \subset \mathbb{P}^2$ be lines over $k$, in general position. Let $\epsilon > 0$. Then there exists a finite union of lines $Z \subset X $ such that for all points $P \in \mathbb{P}^2(\overline{k}) \setminus Z$ satisfying $[k(P):k]\leq 2$ we have
$$
\sum_{v \in S}\sum_{\substack{w\mid v\\ w \in M_{k(P)}}} \sum_{i=1}^q {\lambda_w(L_{i},P)}<\left(\dfrac{15}{2}+\epsilon \right)h(P).
$$ 
\end{Lemma}
\section{Preliminaries on correspondences}\label{corresp}
 Let $X$ be an algebraic variety defined over
some field $k$, and let $\tilde{C}$ be a correspondence on $X$, by which we mean a subvariety
$\tilde{C} \subset X \times X$ such that the coordinate projections are both finite and surjective. We will follow the construction of Ingram here \cite{I} (see also Autissier \cite{A}).

We  think of the points $X(k)$ as a set of vertices, with the points $\tilde{C}(k)$
defining directed edges between those vertices.  There is a \textit{path space} $\pi: \mathcal{P} \rightarrow X$ parametrizing paths through
the aforementioned directed graph, with initial vertex indicated by $\pi$, as well as a shift map $\sigma : \mathcal{P} \rightarrow \mathcal{P}$ which corresponds to stepping forward once
along the path (or, equivalently, forgetting the first vertex of a path). This is described in the first lemma below.

Let $x, y : \tilde{C} \rightarrow X$ be the coordinate projections. Ingram \cite{I} studies polarized correspondences. Namely, he defined that the correspondence $\tilde{C}$ is polarized if and only if there exists an ample Cartier divisor $L \in Div(X)\otimes \mathbb{R}$ and a real number $\alpha> 1$ such that $y^*L$ is linearly equivalent to $\alpha x^*L$. It is easy to check that in the case where $\tilde{C}$ is the graph of a morphism $f : X \rightarrow X$, the correspondence is polarized in case the algebraic dynamical system is polarized in the usual sense. He proved the following important construction.

\begin{Lemma}\cite[Theorem 2.1]{I}  Let $X$ be a separated, integral $S$-scheme of finite type and let $\tilde{C}$ be a correspondence in the $S$-scheme $X$ (a subscheme of $X \times_S X$ such that the two projections are finite and surjective). Then there is a separated, integral, $S$-scheme $\mathcal{P}$, surjective morphisms  $\pi: \mathcal{P} \rightarrow X, \epsilon: \mathcal{P} \rightarrow \tilde{C}$ and a finite surjective morphism $\sigma:\mathcal{P} \rightarrow \mathcal{P}$ such that $\pi=x\circ \epsilon$ and $\pi \circ \sigma=y\circ \epsilon.$
\end{Lemma}
Now, we define a generalization of a pullback of an iterated map for the context of correspondences that will allow us to generalize integrality results to this setting. 
\begin{Def}
 Let $D$ be a divisor of $X$ and $\tilde{C}$ a correspondence on $X$ with coordinate projections $x, y : \tilde{C} \rightarrow X$. We say that a divisor $D^{(n)}$ of $X$ is an $n$-th \textit{Correspondence Iterated Pullback (CIP) of $D$ with respect to $\tilde{C}$}  if there are divisors $L_1,...,L_{m-1}$ of $X$ such that, making $L_0=D$ and $L_m=D^{(m)}$, we have 
 $$
 y^*L_{i-1}= x^*L_i \text{ for all } i=1,...,m.
 $$
\end{Def} In particular, when $\tilde{C}$ is the graph of a morphism $f:X\rightarrow X$, the CIP definition recovers the usual iterated pullback classes $D^{(m)}=(f^{(m)})^*D$, and when the correspondence is polarized in the sense of Ingram, then, modulo $\mathbb{R}$-linear equivalence, the CIP iterates will have the expected $\alpha^m$-scaling behaviour.

For our results, we will generally assume that an effective divisor $D$ of $X$ has an $m$-th (CIP) $D^{(m)}$  with respect to  a correspondence $\tilde{C}$ such that $D^{(m)}=D_1+...+D_q +D'$ is a sum of effective divisors where $D_1,...,D_q$ are in subgeneral position.

\section{Proofs and results}
\subsection{Sparsity of integral points in orbits}
Here we present some new results in light of the new diophantine approximation result of He and Ru.

Moreover, we prove results in semigroup dynamics. Namely, for $\phi_1,...,\phi_\ell:X\rightarrow X$ not necessarily commuting, we consider $\mathcal{F}:=<\phi_1,...,\phi_\ell>$ the semigroup generated by $\phi_1,...,\phi_\ell$ with the composition of functions as its operation.
\begin{Th}\label{thm:evefercons}
Let $X$ be a projective variety of dimension $n$, $S \subset$ be a finite set of places, $\phi_1, \ldots, \phi_\ell$ be endomorphisms of $X$, and $\mathcal F = \langle \phi_1, \ldots, \phi_\ell \rangle$ all def. over a number field $k$.  Suppose there exist $ \psi_1, \ldots, \psi_m\in \mathcal F$, a nontrivial effective divisor $D$ and $A$ an ample divisor of $X$ def. over $k$  such that
\begin{itemize}
\item[\emph{(i)}] $\psi_i^*(D) = D_{i1}  + D_{i2} + \cdots +D_{i, q_i} + D_{i}'$ is a sum of effective divisors def. over $k$, where $D_{i1}, \ldots, D_{i, q_i}$ are in $l_i$-subgeneral position for some $l_i\geq n$.
\item[\emph{(ii)}] There exist $c_{ij} \in \mathbb{Q}^{>0}$ such that
$ \sum_{j=1}^{q_i} c_{ij}{D_{ij}} - (l_i-n+1)(n+1) A$ is big and $A-c_{ij}D_{ij}$ is $\mathbb{Q}$-nef for all $i,j$.

\end{itemize}
Then for all sufficiently small $\epsilon >0$   the set
\begin{equation}\label{HeRu}
\left\{\phi(P): P \in X(k),  \phi\in \bigcup_{i=1}^m (\psi_i \circ\mathcal F), \sum_{v\notin S} \lambda_v(D, \phi(P)) \le \epsilon h(A,\phi(P))\right\}
\end{equation}is not Zariski dense.
\end{Th}
\begin{proof}
Suppose that $\phi(P)$ is in the set \eqref{HeRu}.  We can suppose that $\phi = \psi_i \circ \eta$ for some $\eta \in \mathcal F$ and $1\le i\le m$.  By functoriality of heights,
\begin{equation}\begin{split}
\sum_{v\notin S} &\lambda_v(D, \phi(P)) = h(D, \psi_i(\eta(P))) - \sum_{v\in S} \lambda_v(D,\psi_i (\eta(P)))+O(1) \\
&= h(\psi_i^*(D), \eta(P))  - \sum_{v\in S} \lambda_v(\psi_i^*(D), \eta(P))+O(1) \\
&= \sum_{v\notin S} \lambda_v(\psi_i^*(D), \eta(P))+O(1) \\
&\ge \sum_{j=1}^{q_i} \sum_{v\notin S} \lambda_v(D_{ij}, \eta(P))+O(1) \label{eq1} \end{split}
\end{equation}
 on a Zariski-dense subset of $X(k)$.  On the other hand, since $L_i:=\sum_{j=1}^{q_i} c_{ij}{D_{ij}} - (l_i-n+1)(n+1) A$ is big, Lemma 2.1 implies there exists a proper Zariski-closed set $Z'_i$ and a positive constant $c_i'$ such that
\begin{equation}\label{eq2}
h(L_i, Q) \ge  c_i' h(A,Q)-O(1)
\end{equation}
for all $Q\notin Z'_i$.  In addition, because of ampleness and 
\begin{equation}\label{eq3}
h(A,\psi_i(Q))= h(\psi_i^*(A),Q) +O(1)
\end{equation}
for all $Q$ (functoriality), Lemma 2.1 also implies there exists $c_i'' \ge 1$ not depending on $P$ such that
\begin{equation}\label{eq4}
h(A,\psi_i(\eta(P))) \le c_i'' h(A,\eta(P))+O(1).
\end{equation}
  Therefore,
applying the above with $Q = \eta(P)$, combining with \ref{eq2}, \ref{eq3}, \ref{eq4}, and the condition in \ref{HeRu}, we obtain
\begin{align*}
& h (-(l_i-n+1)(n+1) A, \eta(P)) + \sum_{j=1}^{q_i} \sum_{v\in S} {\lambda_v(c_{ij}D_{ij}, \eta(P))}\\
&= h\left(- (l_i-n+1)(n+1) A + \sum_{j=1}^{q_i} c_{ij}{D_{ij}}, \eta(P)\right) - \sum_{j=1}^{q_i} \sum_{v\notin S} {\lambda_v(c_{ij}D_{ij}, \eta(P))}+O(1)\\
&\ge  c_i' h(A,\eta(P)) -  \left(\max_{j} c_{ij}\right) \sum_{v\notin S} \lambda_v(D, \phi(P))-O(1)\\
&\ge  c_i' h(A,\eta(P)) - \epsilon\cdot  \left(\max_{j} c_{ij}\right) h(A,\phi(P))-O(1)\\
&\ge c_i' h(A,\eta(P)) - \epsilon{c_i''} \max_{j} c_{ij}h(A,\eta(P))-O(1)\\
&=(c_i' -\epsilon{c_i''} \max_{j} c_{ij})h(A,\eta(P))-O(1).
\end{align*}
Choosing
\begin{equation}\label{eq:epsiloncond}
\epsilon < \displaystyle \min_i \left(\frac {c_i'}{c_i''} \min_j 1/c_{ij}\right)
\end{equation}implies $\delta:=\displaystyle\min_i(c_i' -\epsilon{c_i''} \max_{j} c_{ij})>0$.
Applying Lemma 3.1 to the data $A, D_{i1}, D_{i2},...,D_{i, q_i}$, to $ c_{i1}, c_{i2},...,c_{i, q_i}$ and $\delta/2>0$, we have that there exists a union of proper Zariski closed set  $Z_i$ such that $\eta(P) \in Z_i'$, or $\eta(P)\in Z_i$ or it belongs to a finite set  of bounded height from above by $O(1/\delta)$ by the Northcott property. Therefore, this argument shows that if $\phi(P)$ is in the set (\ref{HeRu}), then $\phi(P)= \psi_i(\eta(P)) \in \psi_i(Z'_i) \cup \psi_i(Z_i)$ or to a finite set. By applying these arguments for each $i=1,...,m$, we see that there is a sufficiently small $\epsilon>0$ such that any point in the set (5.1) is contained in the union of a finite set and $\bigcup_{i=1}^m   \psi_i(Z'_i) \cup \psi_i(Z_i)$, which is a proper closed set because it is a union of proper sets  due to the $\psi_i$'s being morphisms between projective varieties, which are closed maps.
\end{proof}

\begin{Cor}
Let $X$ be a projective variety of dimension $n$, $S$ be a finite set of places, $\phi_1, \ldots, \phi_\ell$ be endomorphisms of $X$, and $\mathcal F = \langle \phi_1, \ldots, \phi_\ell \rangle$ all def. over a number field $k$.  Suppose there exist $ \psi_1, \ldots, \psi_m\in \mathcal F$, a nontrivial effective divisor $D$ and $A$ an ample divisor of $X$ def. over $k$  such that
\begin{itemize}
\item[\emph{(i)}] $\psi_i^*(D) = D_{i1}  + D_{i2} + \cdots +D_{i, q_i} + D_{i}'$ is a sum of effective divisors def. over $k$, $D_{i1}, \ldots, D_{i, q_i}$ are in $l_i$-subgeneral position for some $l_i\geq n$.
\item[\emph{(ii)}] $ \sum_{j=1}^{q_i} \frac{D_{ij}}{d_{ij}} - (l_i-n+1)(n+1) A$ is big for all $i$ and $D_{ij}\equiv d_{ij}A $ where $d_{ij}$ are positive integers  for all $ i.$.

\end{itemize}
Then for all sufficiently small $ \epsilon>0$  the set:
\begin{equation*}
\left\{\phi(P): P \in X(k), \phi\in \bigcup_{i=1}^m (\psi_i \circ\mathcal F), \sum_{v\notin S} \lambda_v(D, \phi(P)) \le \epsilon h(A,\phi(P))\right\}
\end{equation*} is not Zariski dense
\end{Cor}\begin{proof}
This is a consequence of the previous theorem applied with $c_{ij}=1/d_{ij}.$
\end{proof} Both the theorem above and its corollary are unconditional statements in the spirit of \cite[Thm. 4.2]{MY} which assumed Vojta's Conjecture.

\begin{Cor}
Let $X$ be a projective variety of dimension $n$, $S$ be a finite set of places and $\phi$ be an endomorphism of $X$ all def. over a number $k$.  Suppose there exist $ m\geq 0$, a nontrivial effective divisor $D$ and $A$ an ample divisor of $X$ def. over $k$  such that
\begin{itemize}
\item[\emph{(i)}] $(\phi^{(m)})^*(D) = D_{1}  + D_{2} + \cdots +D_{ q} + D'$  is a sum of effective divisors def. over $k$, $D_{1}, \ldots, D_{q}$ are in $l$-subgeneral position for some $l\geq n$.
\item[\emph{(ii)}] $ \sum_{j=1}^{q} \frac{D_{j}}{d_{j}} - (l-n+1)(n+1) A$ is big  and $D_{j}\equiv d_{j}A $   where $d_{j}$ are positive integers for all $j$.
\end{itemize}
Then for all sufficiently small $ \epsilon>0$  the set
\begin{equation*}
\left\{\phi^{(n)}(P): P \in X(k),  n \geq m, \sum_{v\notin S} \lambda_v(D, \phi^{(n)}(P)) \le \epsilon h(A,\phi^{(n)}(P))\right\}
\end{equation*} is not Zariski dense.

\end{Cor}

\subsection{Algebraic points of bounded degree in orbits}
Here we prove some new statements on sparsity of integral points of bounded degree in orbits, that are also uniform in the sense that they don't depend of the initial orbit points. 
\begin{Th}\label{boundeddegree}
Let $X$ be a projective variety of dimension $n$, $S$ be a finite set of places, $\phi_1, \ldots, \phi_\ell$ be endomorphisms of $X$, and $\mathcal F = \langle \phi_1, \ldots, \phi_\ell \rangle$ all def. over a number field $k$, and let $\delta \in \mathbb{N}^{>0}$.  Suppose there exist $ \psi_1, \ldots, \psi_m\in \mathcal F$, a nontrivial effective divisor $D$ and $A$ an ample divisor of $X$ def. over $k$  such that
\begin{itemize}
\item[\emph{(i)}] $\psi_i^*(D) = D_{i1} + \cdots +D_{i, q_i} + D_{i}'$ is a sum of effective divisors def. over $k$, where $D_{i1}, \ldots, D_{i, q_i}$ are ample and in $l_i$-subgeneral position with $l_i \geq n$.
\item[\emph{(ii)}]$ \sum_{j=1}^{q_i} \frac{D_{ij}}{d_{ij}} - C(l_i,n,\delta) A$ is big  and $D_{ij}\equiv d_{ij}A $  where $d_{ij}$ are positive integers  for all $ i,j$.
\end{itemize}
Then for all sufficiently small $ \epsilon>0$   the set
\begin{equation}\label{Lev}
\left\{\phi(P): P \in X(\overline{k}), [k(P):k]\leq \delta, \phi\in \bigcup_{i=1}^m (\psi_i \circ\mathcal F),  \sum_{v\notin S} \sum_{\substack{w\mid v\\ w \in M_{k(P)}}} \lambda_w(D, \phi(P)) \le \epsilon h(A,\phi(P))\right\}
\end{equation} is not Zariski dense.

\end{Th}
\begin{proof} The proof follows the lines of the proof of Theorem 5.1, adapted to points in extensions of bounded degree.
Suppose that $\phi(P)$ is in the set \eqref{Lev}.  Then $\phi = \psi_i \circ \eta$ for some $\eta \in \mathcal F$ and $1\le i\le m$.  By functoriality of heights,
\begin{equation}\begin{split}
\sum_{v\notin S}  \sum_{\substack{w\mid v\\ w \in M_{k(P)}}}&\lambda_w(D, \phi(P)) = h(D, \psi_i(\eta(P))) - \sum_{v\in S}  \sum_{\substack{w\mid v\\ w \in M_{k(P)}}} \lambda_w(D,\psi_i (\eta(P)))+O(1) \\
&= h(\psi_i^*(D), \eta(P))  - \sum_{v\in S}  \sum_{\substack{w\mid v\\ w \in M_{k(P)}}} \lambda_w(\psi_i^*(D), \eta(P))+O(1) \\
&= \sum_{v\notin S}  \sum_{\substack{w\mid v\\ w \in M_{k(P)}}} \lambda_w(\psi_i^*(D), \eta(P))+O(1) \\
&\ge \sum_{j=1}^{q_i} \sum_{v\notin S}   \sum_{\substack{w\mid v\\ w \in M_{k(P)}}}\lambda_w(D_{ij}, \eta(P)) \label{eq1}+O(1).\end{split}
\end{equation}
  On the other hand, since $L_i:=\sum_{j=1}^{q_i} \frac{D_{ij}}{d_{ij}} - C(l_i,n,\delta) A$ is big, Lemma 2.1 implies there exists a proper Zariski-closed set $Z'_i$ and a positive constant $c_i'$ such that
\begin{equation}\label{eq6}
h(L_i, Q) \ge  c_i' h(A,Q)-O(1)
\end{equation}
for all $Q\notin Z'_i$.  In addition, because of ampleness and 
\begin{equation}\label{eq7}
h(A,\psi_i(Q))= h(\psi_i^*(A),Q) +O(1)
\end{equation}
for all $Q$ (functoriality), Lemma 2.1 also implies there exists $c_i'' \ge 1$ such that
\begin{equation}\label{eq8}
h(A,\psi_i(\eta(P))) \le c_i'' h(A,\eta(P))+O(1)
\end{equation} for all $P$.
  Therefore,
applying the above with $Q = \eta(P)$, combining with \eqref{eq6}, \eqref{eq7}, \eqref{eq8} and the condition in \eqref{Lev}, we obtain
\begin{align*}
& h (-C(l_i,n,\delta) A, \eta(P)) + \sum_{j=1}^{q_i} \sum_{v\in S} \sum_{\substack{w\mid v\\ w \in M_{k(P)}}} {\lambda_w \left(\frac{D_{ij}}{d_{ij}}, \eta(P)\right)}\\
&= h\left(- C(l_i,n,\delta) A + \sum_{j=1}^{q_i} \frac{D_{ij}}{d_{ij}}, \eta(P)\right) - \sum_{j=1}^{q_i} \sum_{v\notin S} \sum_{\substack{w\mid v\\ w \in M_{k(P)}}} {\lambda_w\left(\frac{D_{ij}}{d_{ij}}, \eta(P)\right)}+O(1)\\
&\ge  c_i' h(A,\eta(P)) -  \left(\max_{j} 1/d_{ij}\right) \sum_{v\notin S}\sum_{\substack{w\mid v\\ w \in M_{k(P)}}}  \lambda_w(D, \phi(P))-O(1)\\
&\ge  c_i' h(A,\eta(P)) - \epsilon\cdot  \left(\max_{j} 1/d_{ij}\right) h(A,\phi(P))-O(1)\\
&\ge c_i' h(A,\eta(P)) - \epsilon c_i''\max_{j} \frac{1}{d_{ij}}h(A,\eta(P))-O(1)\\
&=\left(c_i'  - \epsilon c_i''\max_{j} \frac{1}{d_{ij}}\right)h(A,\eta(P))-O(1).
\end{align*}
Choosing 
\begin{equation}\label{eq:epsiloncond}
\epsilon < \displaystyle \min_i \left(\frac {c_i'}{c_i''} \min_j d_{ij}\right)
\end{equation}implies $\Delta:=\displaystyle\min_i(c_i' -\epsilon{c_i''} \max_{j} 1/d_{ij})>0$.
Applying Lemma 3.2 to the data $A, D_{i1}, D_{i2},...,D_{i, q_i}$,  and $\Delta/2>0$, we have that  $\eta(P) \in Z_i'$, or $\eta(P)\in E_i=\bigcup_{j=1}^{q_i}\text{Supp}(D_{ij})$ or it belongs to a finite set  of bounded height from above by $O(1/\Delta)$ by the Northcott property. Therefore, this argument shows that if $\phi(P)$ is in the set (5.7), then $\phi(P)= \psi_i(\eta(P)) \in \psi_i(Z'_i) \cup \psi_i(E_i)$ or to a finite set. By applying these arguments for each $i=1,...,m$, we see that there is a sufficiently small $\epsilon>0$ such that any point in the set (5.7) is contained in the union of a finite set and $\bigcup_{i=1}^m  {\psi_i(Z'_i)} \cup {\psi_i(E_i)}$, which is a proper closed set because it is a union of proper sets  due to the $\psi_i$'s being closed maps since they are morphisms between projective varieties.

\end{proof}
\begin{Th}
Let  $\phi_1, \ldots, \phi_\ell$ be endomorphisms of $\mathbb{P}^2$, and $\mathcal F = \langle \phi_1, \ldots, \phi_\ell \rangle$ all def. over a number $k$ with a finite set of places $S$.  Suppose there exist $\psi_1, \ldots, \psi_m\in \mathcal F$ and a nontrivial effective divisor $D$ def. over $k$ such that
\begin{itemize}
\item[\emph{(i)}] $\psi_i^*(D) = L_{i1}  + L_{i2} + \cdots +L_{i, q_i} + D_{i}'$  a sum of effective divisors def. over $k$, $L_{i1}, \ldots, L_{i, q_i}$ are lines in general position, and
$q_i>15/2$ for all $i$.

\end{itemize}
 Then for all sufficiently small  $ \epsilon>0$  the set
\begin{equation*}
\left\{\phi(P): P \in X(\overline{k}), [k(P):k]\leq 2, \phi\in \bigcup_{i=1}^m (\psi_i \circ\mathcal F),  \sum_{v\notin S} \sum_{\substack{w\mid v\\ w \in M_{k(P)}}} \lambda_w(D, \phi(P)) \le \epsilon h(\phi(P))\right\}
\end{equation*} is not Zariski dense (here, h is the standard logarithmic height on $\mathbb{P}^2$).
\end{Th}
\begin{proof}
We can prove this similarly as in the proof of the previous theorem, adapting accordingly. In fact, we consider $X=\mathbb{P}^2$, $D_{ij}=L_{ij}$ to be lines in general position, and $A\sim O_{\mathbb{P}^2}(1)$ a hyperplane divisor representing a generator of the Picard group of $\mathbb{P}^2$, whose associated Weil height is the usual naive logarithmic height $h$ of $\mathbb{P}^2$ in the statement of the Theorem 5.4. In this case, the condition (ii) from Thm. 5.4 is satisfied with $d_{ij}=1$ and $q_i> 15/2$, which is equivalent  to having $\sum_{j=1}^{q_i}L_{ij}-(15/2) O_{\mathbb{P}^2}(1) $ big. With these data, the proof follows the same steps of the proof of Thm. 5.4 but replacing each $C(l_i,n,\delta)$ with 15/2 and using Lemm 3.3 instead of Lemma 3.2.
\end{proof}
\begin{Cor}
Let  $\phi$ be an endomorphism of $\mathbb{P}^2$ defined over a number $k$.  Suppose there exist $ m\geq 0$ and a nontrivial effective divisor $D$  such that
\begin{itemize}
\item[\emph{(i)}] $(\phi^{(m)})^*(D) = L_{1}  + L_{2} + \cdots +L_{ q} + D'$ a sum of effective divisors, $L_{1}, \ldots, L_{q}$ are lines in general position all def. over $k$, and
$q>15/2$ 
\end{itemize}
Then for all sufficiently small $ \epsilon>0$ the set
\begin{equation*}
\left\{\phi^{(n)}(P): P \in X(\overline{k}),  [k(P):k]\leq 2, n \geq m, \sum_{v\notin S}\sum_{\substack{w\mid v\\ w \in M_{k(P)}}} \lambda_w(D, \phi^{(n)}(P)) \le \epsilon h(\phi^{(n)}(P))\right\}
\end{equation*}is not Zariski dense.

\end{Cor}

\subsection{Generalization to Correspondences}
We state results about integral points in correspondences now which apply under more general conditions satisfied by some correspondence iterated pullbacks (CIP) (see definition 4.2), first over a number field. We use the following notation.
\begin{Th}
Let $X$ be a smooth projective variety of dimension $n$ defined over a number field $k$,  $A$ an ample divisor of $X$, $S$ be a finite set of places of $k$, and $\tilde{C} \subset X \times X$ be a subvariety (correspondence) of $X$ defined over $k$. Let $\mathcal P$ be the scheme of paths associated to $\tilde{C}$, with associated maps $\pi : \mathcal{P} \rightarrow X$ and  $\sigma: \mathcal{P} \rightarrow \mathcal{P}$ as it is described in Section \ref{corresp}.  

Suppose that there exist $m \in \mathbb{N}^{>0}$ and a nontrivial effective divisor $D$ of $X$ defined over $k$ such that
\begin{itemize}
\item[\emph{(i)}] $D$ has an $m$-th iterate (CIP)   $D^{(m)} = D_1+...+D_q + D^{(\text{rest})}$ w.r.t. $\tilde{C}$ that is a sum of effective divisors def. over $k$ where $D_1,...,D_q$ are in $l$-subgeneral position for some $l\geq n$.

\item[\emph{(ii)}] There exists $c_j>0 \in \mathbb{Q}$ such that
$ \sum_{j=1}^{q} c_{j}{D_{j}} - (l-n+1)(n+1) A$ is big  and   $A-c_{j}D_{j}$ is $\mathbb{Q}$-nef for all $1 \leq  j \leq q$.

\end{itemize}
Then for all sufficiently small $\epsilon>0$, the set
\begin{equation}\label{eq9}
\left\{
\begin{aligned}
\pi(\sigma^{(r)}(P)) \in X(k):\quad&
P\in\mathcal{P}(\overline{k}),\quad
\pi(\sigma^{(r-m)}(P)))\in X(k) ,\quad r\geq m,\\
&\sum_{v\notin S}
\lambda_v\bigl(D,\pi(\sigma^{(r)}(P))\bigr)
\leq\epsilon h\bigl(A,\pi(\sigma^{(r)}(P))\bigr)
\end{aligned}
\right\}
\end{equation}
is not Zariski dense in $X$.
\end{Th}
\begin{proof}
Suppose $\pi(\sigma^{(r)}(P))$ belongs to the set \eqref{eq9}. We can suppose $r\geq  m$. The hypothesis (i) means that there exist divisors $L_0,...,L_m$ of $X$, with $D=L_0$ and $L_m=D^{(m)}$ such that $y^*L_{i-1}= x^*L_i$ for each $i=1,...,m.$ Let $v$ be an arbitrary place of $k$ and let $w$ be an extension of $v$ to the smallest extension of $k$ containing $\pi(\sigma^{(r-1)}(P)),\pi(\sigma^{(r-2)}(P)),...,\pi(\sigma^{(r-m+1)}(P))$. Due to functoriality of heights and Lemma 4.1, this implies  that 
\newline \newline
$\lambda_v(D, \pi(\sigma^{(r)}(P)))=\lambda_v(D, y(\epsilon(\sigma^{(r-1)}(P)))= \lambda_w(y^*D, \epsilon(\sigma^{(r-1)}(P))+O(1)\newline \newline=\lambda_w(x^*L_1, \epsilon(\sigma^{(r-1)}(P))+O(1) = \lambda_w(L_1, x(\epsilon(\sigma^{(r-1)}(P)))+O(1)\newline \newline=\lambda_w(L_1, \pi(\sigma^{(r-1)}(P))+O(1).\newline \newline
$ Repeating the same argument $m$ times, we obtain
$$
\lambda_v(D, \pi(\sigma^{(r)}(P)))= \lambda_w(L_m, \pi(\sigma^{(r-m)}(P)))+O(1)= \lambda_v(L_m, \pi(\sigma^{(r-m)}(P)))+O(1).
$$ By functoriality and Lemma 4.1 the same holds if we replace $\lambda_v(D,.)$ by $h(D,.)$.

By functoriality of heights again,
\begin{equation}\begin{split}
\sum_{v\notin S} &\lambda_v(D, \pi(\sigma^{(r)}(P))) = h(D, \pi(\sigma^{(r)}(P))) - \sum_{v\in S} \lambda_v(D,\pi(\sigma^{(r)}(P)))+O(1) \\
&= h(L_m, \pi(\sigma^{(r-m)}(P)))  - \sum_{v\in S} \lambda_v(L_m, \pi(\sigma^{(r-m)}(P)))+O(1) \\
&= \sum_{v\notin S} \lambda_v(D^{(m)}, \pi(\sigma^{(r-m)}(P)))+O(1) \\
&\ge \sum_{j=1}^{q} \sum_{v\notin S} \lambda_v(D_{j}, \pi(\sigma^{(r-m)}(P)))+O(1) \label{eq1} \end{split}
\end{equation}
 In addition, because of the ampleness of $x^*A$, Lemma 2.1 and  Lemma 4.1, there exists a constant $c \ge 1$ not depending on $Q$ such that \newline \newline
$
h(A,\pi(\sigma^{(m)}(R)))= h(y^*(A),\epsilon(\sigma^{(m-1)}(R)))+O(1) \newline \newline \leq c h(x^*(A),\epsilon(\sigma^{(m-1)}(R)))+O(1)=ch(A,\pi(\sigma^{(m-1)}(R)))+O(1)
$\newline \newline
for all $R \in \mathcal{P}(\overline{k})$ (functoriality). Doing this $m$ times yields the existence of a positive constant $C$ such that
\begin{equation}\label{eq4}
 h(A,\pi(\sigma^{(m)}(R)))\le C h(A,\pi(R))+O(1).
\end{equation}
Also, since $L=\sum_{j=1}^{q} c_j{D_{j}} - (l-n+1)(n+1) A$ is big, Lemma 2.1 also implies  there exists a constant $c'>0$ such that
$$
c' h(A,Q) \leq  h(L, Q)+O(1)
$$ for all $Q \in X$ outside  a proper Zariski-closed set $Z'_1$.

Applying all of the above with $Q = \pi(\sigma^{(r-m)}(P))$  and the condition in \eqref{eq9},  we obtain
\begin{align*}
& h (-(l-n+1)(n+1) A, \pi(\sigma^{(r-m)}(P))) + \sum_{j=1}^{q} \sum_{v\in S} {\lambda_v(c_{j}D_{j},\pi (\sigma^{(r-m)}(P)))}\\
&= h\left(- (l-n+1)(n+1) A + \sum_{j=1}^{q} c_{j}{D_{j}}, \pi(\sigma^{(r-m)}(P))\right) - \sum_{j=1}^{q} \sum_{v\notin S} {\lambda_v(c_{j}D_{j}, \pi(\sigma^{(r-m)}(P)))}-O(1)\\
&\ge h\left( L, \pi(\sigma^{(r-m)}(P))\right) - \left(\max_{j} c_{j}\right)\sum_{v\notin S} {\lambda_v(L_m, \pi(\sigma^{(r-m)}(P)))}-O(1)\\
&\ge  c' h(A,\pi(\sigma^{(r-m)}(P))) -  \left(\max_{j} c_{j}\right) \sum_{v\notin S} \lambda_v(D, \pi(\sigma^{(r)}(P)))-O(1)\\
&\ge  c' h(A,\pi(\sigma^{(r-m)}(P))) - \epsilon\cdot  \left(\max_{j} c_{j}\right) h(A,\pi(\sigma^{(r)}(P)))-O(1)\\
&\ge c' h(A,\pi(\sigma^{(r-m)}(P))) - \epsilon C\max_{j} c_{j}h(A,\pi(\sigma^{(r-m)}(P)))-O(1).
\end{align*}

Therefore, whenever $\epsilon$ is small enough so that $\Delta=c' - \epsilon C\max_j{c_j} >0$, Lemma 3.1 applied to the set of divisors $\{D_j\}_{j \leq q}$ and the number $\Delta/2>0$ implies that there exists a proper Zariski-closed set $Z_1$ such that either $\pi (\sigma^{(r-m)}(P)) \in Z_1$ or $\pi (\sigma^{(r-m)}(P)) \in Z'_1$,  or $h(\pi (\sigma^{(r-m)}(P)))=O(1/\Delta)$. The latter form a set $F$ of bounded height which is finite  by the Northcott property. Thus, making $Z= Z_1 \cup Z_1' \cup F$, we have $\pi (\sigma^{(r-m)}(P)) \in Z,$ which is Zariski closed and proper. Note that, by Lemma 4.1, if $\pi(Q) \in Z,$ then $\pi (\sigma(Q)) \in y (x^{-1}(Z))$. We make $y (x^{-1}(Z))=T(Z).$ This argument shows that if $\pi (\sigma^{(r)}(P))$ is in the set \eqref{eq9}, since we proved that $\pi( \sigma^{(r-m)}(P))\in Z$, this implies that $\pi (\sigma^{(r)}(P))$ is in $T^m(Z):= \underbrace{T(T(\cdots(Z))\cdots))}_{m\text{ times}}.$ This set is a proper Zariski closed set as desired  because $x$ and $y$ are finite projections, and so $\dim y(x^{-1}(Z))<n$ and $\dim T^m(Z)<n$ inductively.
\end{proof} Similar to the deduction of Corollary 5.2, the hypothesis of the previous theorem are satisfied in the next statement, with $c_j=1/d_j$.
\begin{Cor}
Let $X$ be a smooth projective variety of dimension $n$ defined over a number field $k$,  $A$ an ample divisor of $X$, $S$ be a finite set of places of $k$, and $\tilde{C} \subset X \times X$ be a subvariety (correspondence) of $X$ all defined over $k$, with $\sigma$ and $\pi$ as before.

Suppose that there exist $m \in \mathbb{N}^{>0}$ and a nontrivial effective divisor $D$ of $X$ defined over $k$ such that
\begin{itemize}
\item[\emph{(i)}]  $D$ has an $m$-th (CIP)  $D^{(m)} = D_1+...+D_q + D^{(\text{rest})}$ w.r.t. $\tilde{C}$ that is a sum of effective divisors def. over $k$ where $D_1,...,D_q$ are in $l$-subgeneral position for some $l\geq n$.
\item[\emph{(ii)}] $ \sum_{j=1}^{q} \frac{D_{j}}{d_{j}} - (l-n+1)(n+1) A$ is big  and $D_{j}\equiv d_{j}A $   where $d_{j}$ are positive integers for all $j$. 
\end{itemize}
Then for all sufficiently small $\epsilon>0$, the set
\begin{equation}
\left\{
\begin{aligned}
\pi(\sigma^{(r)}(P))  \in X(k):\quad&
P\in\mathcal{P}(\overline{k}),\quad
\pi(\sigma^{(r-m)}(P)))\in X(k): ,\quad r\geq m,\\
&\sum_{v\notin S}
\lambda_v\bigl(D,\pi(\sigma^{(r)}(P))\bigr)
\leq\epsilon h\bigl(A,\pi(\sigma^{(r)}(P))\bigr)
\end{aligned}
\right\}
\end{equation}
is not Zariski dense in $X$. 
In particular, if $l=n=\dim X=1$  and $q>2$, then there exists   $\epsilon>0$ such that the corresponding set above  is finite.

\end{Cor} Now,  we also deduce a type of uniform result for points of correspondences that are quasiintegral and have bounded degree.  For $P=(P_1,P_2,...) \in \mathcal{P}(\overline{k})$ a given path $m$ a positive integer and $r\geq m$, we denote $K(P_r,P_{r-m})$ by $L_{r,m}$, i.e, the field of definition of $P_r$ and ${P_{r-m}}$ over the field $k$.
\begin{Th}
Let $X$ be a smooth projective variety of dimension $n$ defined over a number field $k$,  $A$ an ample divisor of $X$, $S$ be a finite set of places of $k$, and $\tilde{C} \subset X \times X$ be a subvariety (correspondence) of $X$ all defined over $k$ , with $\sigma, \pi, x$ and $y$ as before.  Let also $\delta \in \mathbb{N}^{>0}.$

Suppose that there exist $m \in \mathbb{N}^{>0}$ and a nontrivial effective divisor $D$ of $X$ defined over $k$ such that
\begin{itemize}
\item[\emph{(i)}] $D$ has an $m$-th iterate (CIP)  $D^{(m)} = D_1+...+D_q + D^{(\text{rest})}$ w.r.t. $\tilde{C}$ that is a sum of effective divisors def. over $k$ where $D_1,...,D_q$ are ample and in $l$-subgeneral position for some $l\geq n$.
\item[\emph{(ii)}]  There exist integers $d_i>0$ such that
$ \sum_{j=1}^{q} \frac{D_{j}}{d_{j}} - C(l,n,\delta) A$ is big  and $D_{j}\equiv d_{j}A $  for all $j$. 

\end{itemize}
Then for all sufficiently small $\epsilon>0$, the subset
\[
\left\{
\begin{aligned}
\pi(\sigma^{(r)}(P)):\quad&
P\in\mathcal{P}(\overline{k}),\quad
[k(\pi(\sigma^{(r-m)}(P))):k]\leq\delta,\quad r\geq m,\\
&\sum_{v\notin S}\sum_{\substack{w\mid v\\ w\in M_{k(P_r)}}}
\lambda_w\bigl(D,\pi(\sigma^{(r)}(P))\bigr)
\leq\epsilon h\bigl(A,\pi(\sigma^{(r)}(P))\bigr)
\end{aligned}
\right\}.
\]
is not Zariski dense in $X$.
\end{Th}
\begin{proof} We adapt the proof of Theorem 5.7 to encompass points of bounded degree.
Suppose $\pi(\sigma^{(r)}(P))$ belongs to the set above with $r\geq m$.  
The hypothesis (i) means that there exist divisors $L_0,...,L_m$ of $X$, with $D=L_0$ and $L_m=D^{(m)}$ such that $y^*L_{i-1}= x^*L_i$ for each $i=1,...,m.$ Due to functoriality of heights and Lemma 4.1, we have as is the proof of Thm. 5.7 that
$$
\lambda_w(D, \pi(\sigma^{(r)}(P)))=  \lambda_w(L_m, \pi(\sigma^{(r-m)}(P)))+O(1)
$$ for every place $w$ of the field extension $L_{r,m}$ of $k$ where the points are defined. 
 The same holds if we replace $\lambda_v(D,.)$ with $h(D,.)$

This, similarly as before, will imply that
\begin{equation}\begin{split}
\sum_{v\notin S}\sum_{\substack{w\mid v\\ w \in M_{L_{r,m}}}} \lambda_w(D, \pi(\sigma^{(r)}(P))) &= \sum_{v\notin S}\sum_{\substack{w\mid v\\ w \in M_{L_{r,m}}}}  \lambda_w(D^{(m)}, \pi(\sigma^{(r-m)}(P)))+O(1) \\
&\ge \sum_{j=1}^{q} \sum_{v\notin S}\sum_{\substack{w\mid v\\ w \in M_{L_{r,m}}}}  \lambda_w(D_{j}, \pi(\sigma^{(r-m)}(P)))+O(1) \label{eq1} \end{split}
\end{equation}
 In addition, because of the ampleness of $x^*A$, Lemma 2.1 and  Lemma 4.1, there exists a positive constant $C$ such that
\begin{equation}\label{eq4}
 h(A,\pi(\sigma^{(m)}(R)))\le C h(A,\pi(R))+O(1)
\end{equation} for all $R \in \mathcal{P}(\overline{k}).$
Also, since $L=\sum_{j=1}^{q} \dfrac{D_{j}}{d_j} - C(l,n,\delta) A$ is big, Lemma 2.1 also implies there exists a constant $c'>0$ such that
$$
c'h(A,Q) \leq  h(L, Q)+O(1)
$$ for all $Q \in X$ outside  a proper Zariski-closed set $Z'_1$.

Applying all of the above with $Q = \pi(\sigma^{(r-m)}(P))$  and the condition in \eqref{eq9},  and making $c_j=1/d_j$ we obtain 
\begin{align*}
& h (-C(l,n,\delta)A, \pi(\sigma^{(r-m)}(P))) + \sum_{j=1}^{q} \sum_{v\in S}\sum_{\substack{w\mid v\\ w \in M_{L_{r,m}}}}  {\lambda_w(c_{j}D_{j},\pi (\sigma^{(r-m)}(P)))}\\
&= h\left(- C(l,n,\delta)A + \sum_{j=1}^{q} c_{j}{D_{j}}, \pi(\sigma^{(r-m)}(P))\right) - \sum_{j=1}^{q} \sum_{v\notin S}\sum_{\substack{w\mid v\\ w \in M_{L_{r,m}}}} {\lambda_w(c_{j}D_{j}, \pi(\sigma^{(r-m)}(P)))}-O(1)\\
&\ge h\left(L, \pi(\sigma^{(r-m)}(P))\right) - \left(\max_{j} c_{j}\right)\sum_{j=1}^{q} \sum_{v\notin S}\sum_{\substack{w\mid v\\ w \in M_{L_{r,m}}}} {\lambda_w(L_m, \pi(\sigma^{(r-m)}(P)))}-O(1)\\
&\ge  c' h(A,\pi(\sigma^{(r-m)}(P))) -  \left(\max_{j} c_{j}\right) \sum_{v\notin S}\sum_{\substack{w\mid v\\ w \in M_{L_{r,m}}}}  \lambda_w(D, \pi(\sigma^{(r)}(P)))-O(1)\\
&\ge  c' h(A,\pi(\sigma^{(r-m)}(P))) - \epsilon\cdot  \left(\max_{j} c_{j}\right) h(A,\pi(\sigma^{(r)}(P)))-O(1)\\
&\ge c' h(A,\pi(\sigma^{(r-m)}(P))) -C \epsilon  \max_{j} c_{j}h(A,\pi(\sigma^{(r-m)}(P)))-O(1).
\end{align*}

 Therefore, whenever $\epsilon$ is small enough so that $\Delta=c' - \epsilon C\max_j{c_j} >0$, Lemma 3.2 applied to the set of divisors $\{D_j\}_{j \leq q}$ and the number $\Delta/2>0$ implies that there exists a proper Zariski-closed set $Z_1$ such that either $\pi (\sigma^{(r-m)}(P)) \in Z_1$ or $\pi (\sigma^{(r-m)}(P)) \in Z'_1$,  or $h(\pi (\sigma^{(r-m)}(P)))=O(1/\Delta)$. The latter form a set $F$ of bounded height  which is finite  by the Northcott property. Thus, making $Z= Z_1 \cup Z_1' \cup F$, we have $\pi (\sigma^{(r-m)}(P)) \in Z,$ which is Zariski closed and proper. As in the end of the proof of Theorem 5.7, this implies that $\pi (\sigma^{(r)}(P))$ is in $ T^m(Z):= \underbrace{T(T(\cdots(Z))\cdots))}_{m\text{ times}},$ and this set is a proper Zariski closed set again because $x$ and $y$ are finite projections. 

\end{proof}

\textit{Proof of Theorem 1.4:} We can prove this similarly as in the proof of the previous theorem, adapting accordingly. In fact, we consider $X=\mathbb{P}^N$, $D_{j}=L_{j}$ to be lines in general position, and $A\sim O_{\mathbb{P}^N}(1)$ a hyperplane divisor representing a generator of the Picard group of $\mathbb{P}^N$, whose associated Weil height is the usual naive logarithmic height $h$ of $\mathbb{P}^N$ in the statement of the Theorem 5.9. .With these data, the proof follows the same steps and conclusion of the proof of Thm. 5.9.

Note again that if $N=2$ and the condition, we can repeat the previous proof with the sharper bound 15/2 in place of $C(l,n,\delta)$, using Lemma 3.3 instead of Lemma 3.2. 

Finally, if we apply both Corollary 5.8 and Theorem 5.9 with $X=\mathbb{P}^1$ we obtain the following

\begin{Cor} 
Let  $\tilde{C} \subset \mathbb{P}^1 \times  \mathbb{P}^1$ be a correspondence of $ \mathbb{P}^1$ defined over $k$. Let $\mathcal P$, $\pi, x, y$   and  $\sigma$ as before.  Let $D$ be an effective divisor of $ \mathbb{P}^1$ and suppose that $D$ has an $m$-th iterate (CIP)     $D^{(m)} = D_1+...+D_q + D^{(\text{rest})}$ w.r.t. $\tilde{C}$ that is a sum of effective divisors def. over $k$  where $D_1,...,D_q$ are ample in general position  and $q > C(1,1,\delta)$ .
  
Then for all sufficiently small $\epsilon>0$, the subset
\[
\left\{
\begin{aligned}
\pi(\sigma^{(r)}(P)):\quad&
P\in\mathcal{P}(\overline{k}),\quad
[k(\pi(\sigma^{(r-m)}(P))):k]\leq\delta,\quad r\geq m,\\
&\sum_{v\notin S}\sum_{\substack{w\mid v\\ w\in M_{k(P_r)}}}
\lambda_w\bigl(D,\pi(\sigma^{(r)}(P))\bigr)
\leq\epsilon h\bigl(\pi(\sigma^{(r)}(P))\bigr)
\end{aligned}
\right\}.
\]
 is finite. 
\end{Cor}
\begin{proof}
This is a direct consequence of Theorem 5.9.
\end{proof}
\textit{Proof of Theorem 1.5:} We will apply Theorem 5.9  with $D=(\infty)$. By hypothesis we have $D^{(m)}=\sum_{j=1}^{q}c_j\,(\alpha_j),\text{ for } \alpha_1,\dots,\alpha_q\in\mathbb{P}^1(\overline k)
\text{ distinct},\text{ and } c_j\in\mathbb{Q}_{>0}.$ By enlarging $k$, we can assume $ \alpha_1,\dots,\alpha_q$ are defined over $k$ without loss of generality. Choosing $N \in \mathbb{N}$ such that $Nc_j\geq 1$, we have that $N(\infty)$ has $m$-th CIP
\[
ND^{(m)}=(\alpha_1)+\cdots+(\alpha_q)+D^{\rm rest},\qquad
D^{\rm rest}:=\sum_{j=1}^{q}(Nc_j-1)(\alpha_j).
\]
Here $D^{\rm rest}$ is  effective by the choice of $N$.  Moreover $\lambda_w(N(\infty),\cdot)=N\lambda_w(\infty,\cdot)$, so the
condition $\sum_{v\notin S}\sum_{w\mid v}\lambda_w(\infty,P_r)\le\varepsilon\,h(P_r)$ is the
same as $\sum_{v\notin S}\sum_{w\mid v}\lambda_w(N\infty,P_r)\le N\varepsilon\,h(P_r).$ 

Applying Theorem~5.9 over $k$ with
\[
X=\mathbb{P}^1,\quad n=l=1,\quad A=(\infty),\quad D=N(\infty),\quad D_j=(\alpha_j),\quad d_j=1.
\]
Each $D_j$ is an ample effective divisor defined over $k$ with $D_j\equiv d_jA$; distinct
points of $\mathbb{P}^1$ are in general position ($1$-subgeneral position); and
$\sum_{j=1}^{q}D_j/d_j-C(1,1,\delta)A$ has degree $q-C(1,1,\delta)>0$, hence is big.

\begin{Ex}
Let $k=\mathbb{Q}$, $S=\{\infty,2\}$ and let $C\subset\mathbb{P}^1\times\mathbb{P}^1$ be parametrized by
\[
x=t^2,\quad y=\frac{1}{t}\prod_{j=1}^{5}(t^2-j),\quad \iff \quad
x\,y^2=\Bigl(\prod_{j=1}^{5}(x-j)\Bigr)^2 .
\]
The normalization is $\tilde{C}=\mathbb{P}^1_t$ (as $t=\prod_j(x-j)/y$), $\deg(x)=2$, $\deg(y)=10$. Take $D=2(\infty)$.
Then
\[
y^*D=2(t{=}0)+18(t{=}\infty)=x^*L_1,\quad L_1=(0)+9(\infty), \quad y^*(0)=x^*\bigl((1)+\cdots+(5)\bigr),
\]
and
\[
D^{(2)}=(1)+(2)+(3)+(4)+(5)+\tfrac92(0)+\tfrac{81}2(\infty),\qquad
\deg D^{(2)}=50.
\]
So $q=7>C(1,1,3)=6$, and for small $\varepsilon$ there are only finitely many vertices $P_r$ ($r\ge2$) that are
$\varepsilon$-quasi-integral with $[\mathbb{Q}(P_{r-2}):\mathbb{Q}]\le 3$. Again, $C$ has infinitely many $S$-integral points:
$t=\pm2^n$ gives $\bigl(4^n,\pm\prod_{j=1}^5(4^n-j)/2^n\bigr)$, and again
$L_1$ is supported on only $2$ points, so the second iterate is necessary for the finiteness conclusion.
\end{Ex}



\end{document}